\documentclass[11pt]{article}

\usepackage{amsmath, amsthm, amsopn, amssymb, enumerate}

\newtheorem{thm}{Theorem}[section]
\newtheorem{lemma}[thm]{Lemma}
\newtheorem{prop}[thm]{Proposition}
\newtheorem{claim}[thm]{Claim}
\newtheorem{conj}[thm]{Conjecture}

\theoremstyle{definition}
\newtheorem{dfn}[thm]{Definition}

\newcommand{\Ex}{\mathbb{E}}
\newcommand{\Bin}{\mathrm{Bin}}
\newcommand{\bal}{\mathrm{bal}}

\newcommand{\cA}{\mathcal{A}}
\newcommand{\cB}{\mathcal{B}}
\newcommand{\cF}{\mathcal{F}}
\newcommand{\cP}{\mathcal{P}}

\newcommand{\dnp}{\delta_{n,p}}
\renewcommand{\dh}{{\lfloor \delta(G)/2 \rfloor}}

\newcommand{\dist}{\mathrm{dist}}

\newcommand{\lep}{\preccurlyeq}

\title{Optimal packings of Hamilton cycles in sparse random graphs}
\author{
  Michael Krivelevich\thanks{School of Mathematical Sciences, Tel Aviv University, Tel Aviv~69978, Israel. E-mail address: krivelev@post.tau.ac.il. Research supported in part by a USA-Israel BSF grant and by a grant from the Israel Science Foundation.}
  \and
  Wojciech Samotij\thanks{School of Mathematical Sciences, Tel Aviv University, Tel Aviv~69978, Israel; and Trinity College, Cambridge~CB2~1TQ, UK. E-mail address: samotij@post.tau.ac.il. Research supported in part by ERC Advanced Grant DMMCA.}
}
\date{\today}

\begin{document}

\maketitle

\begin{abstract}
  We prove that there exists a positive constant $\varepsilon$ such that if $\log n / n \leq p \leq n^{-1+\varepsilon}$, then asymptotically almost surely the random graph $G \sim G(n,p)$ contains a collection of $\lfloor \delta(G)/2 \rfloor$ edge-disjoint Hamilton cycles.
\end{abstract}

\section{Introduction}

Hamiltonicity has long been one of the main motives in the theory of random graphs, with great many beautiful and inspiring results obtained over the years. The most central question about the~minimal edge probability $p$, for which the binomial random graph $G(n,p)$ contains a~Hamilton cycle asymptotically almost surely, or a.a.s.~for brevity, has been settled by Bollob\'as~\cite{Bo84} and by Koml\'os and Szemer\'edi~\cite{KoSz}, who proved that if $p = \frac{\log n+\log\log n+\omega(1)}{n}$, where $\omega(1)$ is any function tending to infinity with the number of vertices $n$, then the random graph $G(n,p)$ is a.a.s.~Hamiltonian. The hitting time version of this result was established by Bollob\'as~\cite{Bo84} and by~Ajtai, Koml\'os, and Szemer\'edi~\cite{AjKoSz}.

Once the question about the threshold for the appearance of a Hamilton cycle had been settled, problems about finding many edge-disjoint Hamilton cycles appeared as the natural next goal to~conquer. As every Hamilton cycle consumes exactly two edges at any vertex, one can pack at~most $\dh$ edge-disjoint Hamilton cycles in any graph $G$, where as usual $\delta(G)$ stands for the minimum degree of $G$. A very well known conjecture, stated explicitly in~\cite{FrKr08}, suggests that this trivial upper bound is tight for any value of $p$:

\begin{conj}[{\cite{FrKr08}}]
  \label{conj:main}
  For every $p$ satisfying $0 \leq p(n) \leq 1$, a.a.s.~the random graph $G \sim G(n,p)$ contains $\dh$ edge-disjoint Hamilton cycles.
\end{conj}

There have been several results on the road to establish this conjecture. In the sparse regime, Bollob\'as and Frieze~\cite{BoFr} showed that for every fixed positive integer $k$, if $p(n) \geq \frac{\log n+(2k-1)\log\log n+\omega(1)}{n}$ (which is the minimal $p$ for which $\delta(G(n,p)) \geq 2k$ a.a.s.), then one can typically find $k$ edge-disjoint Hamilton cycles in $G(n,p)$. Frieze and Krivelevich~\cite{FrKr08} showed that Conjecture~\ref{conj:main} holds if $p(n) = (1+o(1)) \log n/n$ (in this range still a.a.s.~$\delta(G(n,p)) \ll np$, where $np$ is a typical vertex degree in $G(n,p)$). Finally, Ben-Shimon, Krivelevich, and Sudakov~\cite{BeKrSu} extended the range of validity of Conjecture~\ref{conj:main} to $p(n) \leq 1.02 \cdot \log n/n$ (here the minimum and the average degree of $G(n,p)$ may be already comparable, but still, a.a.s.$~\delta(G(n,p)) \leq np/300$ for these values of $p$). It should be noted that this very sparse regime is perhaps easier to handle as even after having packed $\dh$ Hamilton cycles in a random graph $G$, most of the edges of $G$ are typically still present, thus making the packing task much easier.

For the dense case, Frieze and Krivelevich~\cite{FrKr05} proved that for every constant $p$, the random graph $G(n,p)$ a.a.s.~contains $(1-o(1))np/2$ edge-disjoint Hamilton cycles, thus establishing the asymptotic version of Conjecture~\ref{conj:main}. This has been extended by Knox, K\"uhn, and Osthus \cite{KnKuOs10} to all $p$ satisfying $p(n) \gg \log n/n$ (note that under such assumption on $p$, a.a.s.~$\delta(G(n,p))=(1-o(1))np$).

A major breakthrough has recently been achieved by Knox, K\"uhn, and Osthus~\cite{KnKuOs11}, who proved, in a rather technically complicated paper, that Conjecture~\ref{conj:main} holds for all $p$ satisfying $\log^{50}n/n \leq p(n) \leq 1-\log^9n/n^{1/4}$. This great result, taken together with the previously obtained results for the sparse case (see~\cite{BeKrSu}) left only the polylogarithmic gap $1.02 \cdot \log n/n \leq p(n) \leq \log^{50}/n$, and also the (perhaps less interesting) very dense range $p(n) \geq 1-\log^9n/n^{1/4}$, where the conjecture still had to be settled.

In this paper, we resolve the polylogarithmic range of Conjecture~\ref{conj:main} by establishing the following theorem.

\begin{thm}
  \label{thm:main}
  There exists a positive constant $\varepsilon$ such that the following is true. Assume that $\log n / n \leq p(n) \leq n^{-1+\varepsilon}$ and $G \sim G(n,p)$. Then $G$ a.a.s.~contains a collection of $\dh$ edge-disjoint Hamilton cycles.
\end{thm}

Theorem~\ref{thm:main}, together with the above stated result of Knox et al.~\cite{KnKuOs11} proves Conjecture~\ref{conj:main} for all $p$ satisfying $p(n) \leq 1-\log^9n/n^{1/4}$. Of course, it would be nice to settle the remaining very dense case as well, however we personally feel that this case is perhaps somewhat less important or attractive, due to the very different nature of very dense random graphs.

It should be noted that our result covers completely the previously established sparse cases~\cite{BeKrSu,BoFr,FrKr08} and overlaps with the result of Knox et al.~\cite{KnKuOs11} for the dense(r) case. The numerical value of $\varepsilon$ in Theorem~\ref{thm:main}, as can be derived from our proofs, is rather small (something like $\varepsilon = 10^{-5}$ would suffice) and can probably be improved substantially through a more careful implementation of the same arguments. However, given the result of~\cite{KnKuOs11}, we found rather little motivation to pursue this~goal. In any case, our approach does not seem to be capable of crossing the value $p(n)=n^{-1/3}$ without making substantial modifications and bringing in new ideas.

Our proof has certain similarities to the previous papers on the subject. Just like in essentially every previous paper, we split the random graph $G \sim G(n,p)$ into several random graphs, using one of them to obtain $\dh$ structures, each close to a Hamilton cycle (in our case these structures will be collections of relatively few vertex-disjoint paths covering all vertices of $G$), and then using the other random pieces to turn these structures into Hamilton cycles sequentially. However, unlike in several previous papers, where permanent-based arguments were used to extract $2$-factors with relatively few cycles (this idea appears to be used first in~\cite{FrKr05} in this context), here we find the desired collections of paths by first splitting the vertex set of $G$ into two nearly equally sized pieces $A_1, A_2$, finding a.a.s.~$k := \dh$ nearly perfect matchings $M_i$ containing no edges crossing between $A_1$ and $A_2$, and then exposing the edges between the parts $A_1,A_2$ and proving that the bipartite graph between them contains typically $k$ nearly perfect matchings $N_i$. Then, juxtaposing these matchings $M_i$ and $N_i$ (independent of each other) for each $i \in \{1, \ldots, k\}$, we argue that the resulting graph of maximum degree $2$ contains typically relatively few paths; this argument has certain similarities to the well-known fact postulating that a random permutation of $n$
elements has few cycles. More details and explanations are provided in the subsequent sections.

The rest of the paper is organized as follows: In Section~\ref{sec:preliminaries}, we collect a few auxiliary results that will be used in the proof of Theorem~\ref{thm:main}: bounds on large deviations of binomial random variables (Section~\ref{sec:lrg-dev}), estimates on the minimum degree of $G(n,p)$ (Section~\ref{sec:minimum-degree}), a sufficient condition for a bipartite graph to contain a regular subgraph (Section~\ref{sec:k-factor}), and a corollary from the famous rotation-extension technique of P{\'o}sa (Section~\ref{sec:Posa}). Finally, we prove Theorem~\ref{sec:preliminaries} in Section~\ref{sec:proof}. A~brief outline of the proof is given in Section~\ref{sec:outline}.

\section{Preliminaries}

\label{sec:preliminaries}

\subsection{Notation}

Let $G$ be a graph. We will denote the vertex and edge sets of $G$ by $V(G)$ and $E(G)$, respectively. The number of edges of $G$ will be denoted by $e(G)$. Given a set $A \subseteq V(G)$, we denote by $G[A]$ the subgraph of $G$ induced by $A$ and abbreviate $e(G[A])$ by $e_G(A)$. For two disjoint sets $A, B \subseteq V(G)$, we will let $e_G(A,B)$ denote the number of edges of $G$ with one endpoint in each of the sets $A$ and $B$. Given a set $A \subseteq V(G)$, we denote by $N_G(A)$ the \emph{external} neighborhood of $A$, i.e., the set of all vertices in $V(G) \setminus A$ that have a neighbor in $A$. The degree of a vertex $v \in V(G)$ will be denoted by $\deg_G(v)$. The minimum degree and the maximum degree of $G$ are denoted by $\delta(G)$ and $\Delta(G)$, respectively. Let $G$ and $H$ be two graphs on the same vertex set $V$. Using $G$ and $H$, we can define two more graphs on $V$: the graph $G \cup H$ with the edge set $E(G) \cup E(H)$ and the graph $G \setminus H$ with the edge set $E(G) \setminus E(H)$. We will sometimes omit the index $G$ in $e_G$, $\deg_G$, and $N_G$ if the graph $G$ is clear from the context.

For a positive integer $n$ and a real $p \in [0,1]$, we will denote by $\Bin(n,p)$ the binomial random variable with parameters $n$ and $p$, i.e., the number of successes in a sequence of $n$ independent Bernoulli trials with success probability $p$. We write $X \sim \Bin(n,p)$ to denote the fact that the random variable $X$ has the same distribution as $\Bin(n,p)$. We always write $\log$ for the natural logarithm. Finally, we remark that we omit all floor and ceiling signs whenever these are not essential.

\subsection{Bounding large deviations}

\label{sec:lrg-dev}

In our proofs, we will often use the following standard estimates for tail probabilities of the binomial distribution, see, e.g., \cite[Appendix A]{AlSp}.

\begin{lemma}
  \label{lemma:lrg-dev}
  Let $n$ be a positive integer, let $p \in [0,1]$ and let $X \sim \Bin(n,p)$.
  \begin{enumerate}[(i)]
  \item 
    \label{item:lrg-dev-1}
    (Chernoff's inequality) For every positive $a$,
    \[
    P(X < np - a) < \exp\left(-\frac{a^2}{2np}\right) \quad \text{and} \quad P(X > np + a) < \exp\left(-\frac{a^2}{2np} + \frac{a^3}{2(np)^2}\right)
    \]
    In particular, if $a \leq np/2$, then
    \[
    P(X > np + a) < \exp\left(-\frac{a^2}{4np}\right).
    \]
  \item
    \label{item:lrg-dev-2}
    For every positive $\kappa$,
    \[
    P(X > \kappa np) \leq \left(\frac{e}{\kappa}\right)^{\kappa np}.
    \]
  \end{enumerate}
\end{lemma}

\subsection{Minimum degree of $G(n,p)$}

\label{sec:minimum-degree}

In this section, we prove some basic estimates on the minimum degree of the binomial random graph $G(n,p)$. For the sake of brevity, we first introduce some notation. For an integer $d$ with $0 \leq d \leq n-1$, let
\[
b(d) = P\left(\Bin(n-1,p) = d\right) \quad \text{and} \quad B(d) = P\left(\Bin(n-1,p) \leq d\right) = \sum_{j = 0}^d b(j).
\]
Moreover, let
\[
\dnp = \min\left\{ d \colon B(d) \geq \log n / n \right\}
\]
and note that if $n$ is sufficiently large and $p \geq \log n /n$, then $\dnp \geq 1$. To see this, observe that for a fixed $n$, the function $p \mapsto \dnp$ is increasing and that if $p = \log n / n$, then $nB(0) = n(1-\log n/n)^{n-1} \to 1$ as $n \to \infty$. A simple calculation shows that
\begin{equation}
  \label{eq:bd-ratio}
  \frac{b(d)}{b(d-1)} = 1 + \frac{np - d}{d(1-p)} \leq 1 + \frac{5}{4}\frac{np-d}{d},
\end{equation}
where the last inequality holds if $p \leq 1/5$. We are now ready to state and prove the main result of this subsection.

\begin{lemma}
  \label{lemma:delta}
  If $\log n / n \leq p \leq n^{-1/2}$, then a.a.s.
  \begin{equation}
    \label{eq:delta}
    np - 2\sqrt{np\log n} \leq \delta(G(n,p)) \leq \dnp \leq np - \frac{1}{2}\sqrt{np\log n}.
  \end{equation}
\end{lemma}
\begin{proof}
  We first prove the lower bound for $\delta(G(n,p))$. Let $v$ be an arbitrary vertex of $G(n,p)$. By Chernoff's inequality (Lemma~\ref{lemma:lrg-dev}~\ref{item:lrg-dev-1}),
  \begin{align*}
    P\left( \deg(v) < np - 2\sqrt{np\log n} \right) & \leq P\left( \Bin(n-1,p) \leq (n-1)p - \frac{3}{2}\sqrt{np\log n} \right) \\
    & \leq \exp\left( -\left(3\sqrt{np\log n}/2\right)^2/(2np) \right) = e^{-9\log n/8} = n^{-9/8}.
  \end{align*}
  Hence, by the union bound, $\delta(G(n,p)) \geq np - 2\sqrt{np\log n}$ with probability at least $1 - n^{-1/8}$.

  For the upper bound, observe first that by the definition of $\dnp$, we have $nB(\dnp) \to \infty$ as $n \to \infty$. It follows from~\cite[Theorem~3.1]{Bo} that a.a.s.~$\delta(G(n,p)) \leq \dnp$. Therefore, it remains to give an upper bound for $\dnp$. To this end, let $m = \lfloor np \rfloor$ and observe that~\eqref{eq:bd-ratio} implies that $b(0) \leq \ldots \leq b(m)$ and that $b(m) \geq \ldots \geq b(n-1)$. Since by Chernoff's inequality (Lemma~\ref{lemma:lrg-dev}~(\ref{item:lrg-dev-1})),
  \[
  6\sqrt{np} \cdot b(m) \geq P\left( \left| \Bin(n-1,p) - (n-1)p \right| < 3\sqrt{np} \right) \geq 1 - 2e^{-4} > 3/4,
  \]
  then we have $b(m) \geq 1/(8\sqrt{np}) \geq 1/(8n^{1/4})$. Finally, if $d$ satisfies
  \[
  \frac{np}{2} \leq np - \frac{1}{2}\sqrt{np \log n} < d \leq m,
  \]
  then by~\eqref{eq:bd-ratio},
  \[
  \frac{b(d)}{b(d-1)} \leq 1 + \frac{5\sqrt{np\log n}}{4np} \leq \exp\left(\frac{5}{4}\sqrt{\frac{\log n}{np}}\right).
  \]
  It follows that
  \[
  b\left(np - \frac{1}{2}\sqrt{np\log n}\right) \geq b(m) \cdot \exp\left(-\frac{5}{4}\sqrt{\frac{\log n}{np}} \cdot \frac{1}{2}\sqrt{np\log n}\right) \geq \frac{n^{-7/8}}{8} \geq \frac{\log n}{n}
  \]
  and hence $np - \frac{1}{2}\sqrt{np \log n} \geq \dnp$. This completes the proof.
\end{proof}

\subsection{Factors in bipartite graphs}

\label{sec:k-factor}

One of the key steps in the proof of Theorem~\ref{thm:main} will be showing that certain bipartite subgraphs of $G(n,p)$ typically contain large regular subgraphs of high degree. Our argument will use the following lemma, which gives a sufficient condition for a balanced bipartite graph to contain a $k$-factor.

\begin{lemma}
  \label{lemma:k-factor}
  Let $G$ be a bipartite graph with color classes $A$ and $B$ with $|A| = |B| = n$. Let $k$ be an integer with $k < \delta(G)$ and let $D = \frac{\Delta(G) - k}{\delta(G) - k}$. If
  \begin{enumerate}[(i)]
  \item
    \label{item:eXY-lower}
    $e(X,Y) \geq k|X||Y|/n$ for all $X \subseteq A$ and $Y \subseteq B$ with $|X|,|Y| \geq n/140$ and
  \item
    \label{item:eXY-upper}
    $e(X,Y) \leq \min\{|X|,|Y|\} \cdot k$ for all $X \subseteq A$ and $Y \subseteq B$ with $|X|,|Y| \leq n/140$ and $1/D \leq |X|/|Y| \leq D$,
  \end{enumerate}
  then $G$ contains a $k$-factor.
\end{lemma}
\begin{proof}
  For an arbitrary pair of sets $X \subseteq A$ and $Y \subseteq B$, let $x$ and $y$ denote the cardinalities of $X$ and $Y$, respectively, and let $\bal(X,Y) = e(X,Y) - k(x+y-n)$. In order to prove that $G$ contains a $k$-factor, it suffices to show that
  \begin{equation}
    \label{eq:bXY}
    \bal(X,Y) \geq 0 \quad \text{for all $X \subseteq A$ and $Y \subseteq B$},
  \end{equation}
  see, e.g., the argument in the proof of \cite[Proposition~3.1]{FrKrLo}. WLOG we may assume that $x + y \geq n$ or otherwise $k(x+y-n) < 0$ and~\eqref{eq:bXY} is trivially satisfied. Moreover, since the assumptions on $G$ are symmetric in $A$ and $B$, we may also assume that $x \leq y$. If $x \geq n/140$, then by assumption~(\ref{item:eXY-lower}), we have
  \begin{equation}
    \label{eq:bXY-case1}
    \bal(X,Y) \geq kxy/n - k(x+y-n) = k/n \cdot (x-n)(y-n) \geq 0.
  \end{equation}
  If $x < n/140$, then we let $Y' = B \setminus Y$ and $y' = |Y'| = n-y \leq x$. Note that
  \begin{align*}
    \bal(X,Y) & = e(X,Y) - k(x+y-n) = e(X,B) - kx - e(X,Y') + ky' \\
    & \geq (\delta(G) - k)x + ky' - e(X,Y')
  \end{align*}
  Now, if $y' \geq x/D$, then \eqref{eq:bXY} follows because in this case $e(X,Y') \leq ky'$ by assumption~(\ref{item:eXY-upper}). Otherwise,
  \[
  \bal(X,Y) \geq (\delta(G) - k)x + (k - \Delta(G))y' > (\delta(G) - k)x - (\Delta(G) - k)x/D = 0. \qedhere
  \]
\end{proof}

\subsection{Boosters and expanders}

\label{sec:Posa}

A key tool in the last part of the proof of Theorem~\ref{thm:main} is the celebrated rotation-extension technique developed by P{\'o}sa~\cite{Po}. In this section, we only state a powerful corollary of this method, around which we will build our argument. We first state two crucial definitions.

\begin{dfn}
  Given an integer $m$ and a positive real $c$, we say that a graph $G$ is an \emph{$(m,c)$-expander} if every subset $U \subseteq V(G)$ with $|U| \leq m$ satisfies $|N_G(U)| \geq c|U|$.
\end{dfn}

\begin{dfn}
  \label{dfn:booster}
  Suppose that a graph $G$ contains a Hamilton path but it is not Hamiltonian. A pair $\{u,v\}$ of vertices of $G$ is called a \emph{booster} if the graph $G \cup \{u,v\}$ is Hamiltonian.
\end{dfn}

\begin{lemma}[{\cite[Corollary~2.10]{KrLuSu}}]
  \label{lemma:Posa}
  Let $m$ be a positive integer, let $G$ be a graph, and let $P$ be a path in $G$. Suppose furthermore that $G$ is an $(m,2)$-expander. Then at least one of the following holds:
  \begin{enumerate}[(i)]
  \item
    $G[V(P)]$ contains a Hamilton path $P'$ whose endpoint has a neighbor outside of $V(P)$,
  \item
    $G[V(P)]$ is Hamiltonian and $|V(P)| \geq m$, or
  \item
    $G[V(P)]$ contains at least $m^2/2$ boosters.
  \end{enumerate}
\end{lemma}

\section{Proof of Theorem~\ref{thm:main}}

\label{sec:proof}

Let $\alpha$, $\beta$, $\lambda$, and $\varepsilon$ be small positive constants satisfying $\varepsilon \ll \lambda \ll \beta \ll \alpha \ll 1$. Assume that $\log n / n \leq p \leq n^{-1+\varepsilon}$, where $n$ is a sufficiently large integer, let $G \sim G(n,p)$, and let $V = V(G)$.


\subsection{Outline}

\label{sec:outline}

The proof of Theorem~\ref{thm:main} will be divided into two independent parts.

\medskip
\noindent
{\bf Part I.}
In Section~\ref{sec:constr-G1}, we will show that a.a.s.~$G$ contains a subgraph $G_1$ with $\delta(G_1) = \delta(G)$ that satisfies the following properties:
\begin{enumerate}[(1)]
\item
  \label{item:outline-1}
  Each subgraph $G'$ of $G_1$ obtained by deleting from $G_1$ a subgraph $H$ with $\Delta(H) \leq \delta(G) - 2$ is a good expander, see Sections~\ref{sec:bounding-S} and~\ref{sec:expander-G1};
\item
  \label{item:outline-2}
  $G_1$ contains a family of pairwise edge-disjoint subgraphs $\cP_1, \ldots, \cP_\dh$, where each $\cP_i$ is a collection of at most $n^{1-\lambda}$ vertex-disjoint paths covering all vertices of $G$, see Section~\ref{sec:paths-in-G1}.
\end{enumerate}
Moreover, we will construct such $G_1$ without revealing all the edges of $G$, so that we have some randomness left in Part II. More precisely, there will be a small set $S \subseteq V$ and $p_2 \approx \beta\sqrt{p\log n / n}$ such that $G$ can be represented as a union of $G_1$ and a binomial random graph $G_2$ on the vertex set $V \setminus S$ with edge probability $p_2$. Even though $S$ will depend on $G_1$, the edges of $G_2$ will not. The proof of property~(\ref{item:outline-1}) is via a sequence of fairly standard (alas, somewhat technical) estimates on the edge distribution in the binomial random graph. To show~(\ref{item:outline-2}), we split the vertex set of $G$ into two sets of equal size, denoted $A_1^1$ and $A_2^1$, prove that a.a.s.~$G$ contains a collection of edge-disjoint matchings $M_1, \ldots, M_\dh$, each of them covering all but at most $o(n^{1-\lambda})$ vertices and using only edges of $G[A_1^1]$ and $G[A_2^1]$, and that the bipartite subgraph of $G$ induced by the pair $(A_1^1, A_2^1)$ contains a $\dh$-regular subgraph $H$ with at least $n - o(n^{1-\lambda})$ vertices. We then randomly decompose $H$ into edge-disjoint matchings $N_1, \ldots, N_\dh$ so that each $N_i$ is contained in a perfect matching $N_i'$ and, crucially, each $N_i'$ is distributed like a uniform random perfect matching in $(A_1^1, A_2^1)$. For each $i$ with $1 \leq i \leq \dh$, we obtain $\cP_i$ by juxtaposing the matchings $M_i$ and $N_i$. Since $N_i$ is contained in a random perfect matching that is independent of $M_i$ (recall that $M_i$ uses only edges of $G[A_1^1]$ and $G[A_2^1]$), we are able to prove that with very high probability, the number of connected components (paths) in $\cP_i$ is at most $n^{1-\lambda}$. Our argument has certain similarities to the well-known fact that a random permutation of $n$ elements has typically very few cycles.

\medskip
\noindent
{\bf Part II.}
In Section~\ref{sec:merging-paths}, we will show that given graphs $G_1$ and $G_2$ as above, a.a.s.~using the edges of $G_2$ we can turn $\cP_1, \ldots, \cP_\dh$ one by one into $\dh$ edge-disjoint Hamilton cycles. This will be achieved by concatenating all paths in each $\cP_i$ using the edges of $G_2$ and finally closing the resulting Hamilton path into a cycle.

\subsection{Constructing the graph $G_1$}

\label{sec:constr-G1}

Let
\[
p_1 = p - \beta\sqrt{\frac{p\log n}{n}}, \quad \text{let} \quad p_2 = 1 - \frac{1-p}{1-p_1},
\]
and note that $p_2 \geq p - p_1 = \beta\sqrt{p\log n / n}$. Let $G_1^* = G(n,p_1)$, let $G_2^* = G(n,p_2)$, and note that $G$ has the same distribution as $G_1^* \cup G_2^*$. The graph $G_1$ will be constructed from $G_1^*$ by adding to it all edges of $G_2^*$ that are incident to vertices of small degree in $G_1^*$; this will guarantee that $\delta(G_1) = \delta(G)$. Recall the definition of $\dnp$ from Section~\ref{sec:minimum-degree}. We let
\[
S = \left\{v \in V \colon \deg_{G_1^*}(v) \leq \dnp + \alpha\sqrt{np\log n} \right\}
\]
and let $G_1$ be the subgraph of $G$ obtained from $G_1^*$ by adding to it all edges of $G_2^*$ that have at least one endpoint in the set $S$. Note that this guarantees that $\deg_{G_1}(v) = \deg_G(v)$ for every $v \in S$. Since a.a.s.~$\delta(G) \leq \dnp$, it follows that a.a.s.~the set $S$ contains all vertices of minimum degree in $G$ and therefore $\delta(G_1) = \delta(G)$. Finally, let $G_2 = G_2^*[V \setminus S]$ and note that the edges of $G_2$ are independent of $G_1$. In the remainder of this section we will prove that a.a.s.~$|S| \leq n^{0.1}$ and $G_1$ satisfies the two properties claimed in the outline of the proof (Section~\ref{sec:outline}).

\subsubsection{Bounding the size of $S$}

\label{sec:bounding-S}

In this section, we will show that the set $S$ is typically very small and that vertices of $S$ are far apart in $G_1$, which will be instrumental in guaranteeing that the graphs $G_1 \setminus H$ are good expanders.

\begin{lemma}
  \label{lemma:S}
  A.a.s., $|S| \leq n^{0.1}$ and there is no path of length at most $4$ in $G_1$ whose (possibly identical) endpoints lie in $S$.
\end{lemma}
\begin{proof}
  We first note that for a fixed vertex $v \in V$,
  \begin{align*}
    P(v \in S) & = P\left( \deg_{G_1^*}(v) \leq \dnp + \alpha\sqrt{np\log n} \right) \leq P\left( \Bin(n-1,p_1) \leq \dnp + \alpha\sqrt{np\log n} + 1 \right) \\
    & \leq P\left( \Bin(n-1, p) \leq \dnp + 2\alpha\sqrt{np\log n} \right) + P\left( \Bin(n-1,p-p_1) \geq \alpha\sqrt{np\log n} \right),
  \end{align*}
  where the above inequality follows from the observation that a random variable with distribution $\Bin(n-1, p)$ is a sum of two random variables with distributions $\Bin(n-1,p_1)$ and $\Bin(n-1,p-p_1)$, respectively. Recall that $p - p_1 = \beta\sqrt{p\log n / n}$ and observe that by Lemma~\ref{lemma:lrg-dev}~(\ref{item:lrg-dev-2}),
  \[
  P\left( \Bin(n-1,p-p_1) \geq \alpha\sqrt{np\log n} \right) \leq \left( \frac{e\beta}{\alpha} \right)^{\alpha\sqrt{np\log n}} \leq e^{-\sqrt{np\log n}} \leq \frac{1}{n},
  \]
  where the first inequality holds because $\beta \ll \alpha$ and the last inequality follows from the assumption that $p \geq \log n / n$. Next, note that
  \begin{equation}
    \label{eq:PvS-one}
    P\left( \Bin(n-1, p) \leq \dnp + 2\alpha\sqrt{np\log n} \right) = B(\dnp - 1) + \sum_{j = 0}^{2\alpha \sqrt{np\log n}} b(\dnp + j).
  \end{equation}
  Recall that $b(\dnp-1) \leq B(\dnp-1) \leq \log n /n$ by the definition of $\dnp$. We claim that the sum in the right-hand side of~\eqref{eq:PvS-one} is bounded by $n^{-0.95}$. To see this, we consider two cases:

  \medskip
  \noindent
  {\bf Case 1.} $p \geq 16\log n / n$. \\
  By Lemma~\ref{lemma:delta} and our assumption on $p$, we have that $\dnp \geq np - 2\sqrt{np\log n} \geq np/2$. By~\eqref{eq:bd-ratio}, for every positive $j$,
  \[
  \frac{b(\dnp + j)}{b(\dnp + j - 1)} \leq 1 + \frac{5}{4}\frac{np - \dnp - j}{\dnp + j} \leq 1 + \frac{5}{4}\frac{np - \dnp}{\dnp} \leq 1 + 5\sqrt{\frac{\log n}{np}} \leq \exp\left( 5\sqrt{\frac{\log n}{np}} \right)
  \]
  and therefore
  \begin{align*}
    \sum_{j=0}^{2\alpha\sqrt{np\log n}} b(\dnp+j) & \leq b(\dnp-1) \cdot \sum_{j=0}^{2\alpha\sqrt{np\log n}} \exp\left( 5\sqrt{\frac{\log n}{np}} \right)^{j+1} \\
    & \leq \frac{\log n}{n} \cdot 3\alpha\sqrt{np\log n} \cdot \exp(11\alpha\log n) \leq \frac{(\log n)^{3/2}}{n^{1-\varepsilon/2-11\alpha}} \leq n^{-0.95},
  \end{align*}
  where the last inequality follows form our assumptions that $p \leq n^{-1+\varepsilon}$ and $\alpha, \varepsilon \ll 1$.

  \medskip
  \noindent
  {\bf Case 2.} $p < 16\log n / n$. \\
  In this case, we recall that $\dnp \geq 1$ and estimate somewhat differently. If $1 \leq j \leq 2\alpha\sqrt{np\log n}$, then
  \[
  \frac{b(\dnp + j)}{b(\dnp + j - 1)} \leq 1 + \frac{5}{4}\frac{np - \dnp - j}{\dnp + j} \leq 1 + \frac{5np}{4j} \leq \frac{2np}{j}.
  \]
  It follows that for every $j$ with $1 \leq j \leq 2\alpha\sqrt{np\log n}$,
  \[
  b(\dnp + j - 1) \leq b(\dnp-1) \cdot \frac{(2np)^j}{j!} \leq \frac{\log n}{n} \cdot \left( \frac{2enp}{j} \right)^j,
  \]
  where the last inequality follows from the well-known inequality $j! > (j/e)^j$. If $j \leq 2\alpha\sqrt{np\log n}$, then by our assumption on $p$ and the fact that $\alpha \ll 1$, it follows that
  \[
  \left( \frac{2enp}{j} \right)^j \leq \left( \frac{e\sqrt{np}}{\alpha\sqrt{\log n} }\right)^{2\alpha\sqrt{np\log n}} \leq \left(\frac{4e}{\alpha}\right)^{8\alpha \log n} \leq n^{0.04}
  \]
  and hence, since $\varepsilon \ll 1$,
  \[
  \sum_{j=0}^{2\alpha\sqrt{np\log n}} b(\dnp+j) \leq \frac{\log n}{n} \cdot 3\alpha\sqrt{np\log n} \cdot n^{0.04} \leq \frac{(\log n)^{3/2}}{n^{1-0.04-\varepsilon/2}} \leq n^{-0.95}.
  \]
  We conclude that
  \begin{equation}
    \label{eq:PvS-two}
    P(v \in S) \leq P\left( \Bin(n-1,p_1) \leq \dnp + \alpha\sqrt{np\log n} + 1 \right) \leq 2n^{-0.95}
  \end{equation}
  and hence by Markov's inequality, $|S| \leq n^{0.1}$ with probability $1 - o(1)$.

  \medskip
  
  Finally, we show that a.a.s.~$G_1$ contains no path of length at most $4$ whose distinct endpoints lie in $S$. The case of identical endpoints is similar (and somewhat simpler) and we leave it as an exercise to the reader. Fix an integer $r$ with $1 \leq r \leq 4$ and let $u, v$ be two distinct vertices of $G$. Let $P$ be a sequence $w_0, \ldots, w_r$ of vertices such that $w_0 = u$ and $w_r = v$ and let $\cB_P$ denote the event that $w_i w_{i+1}$ is an edge of $G_1$ for every $i \in \{0, \ldots, r-1\}$. Clearly, $P(\cB_P) \leq p^r$ and
  \begin{equation}
    \label{eq:uv-short-path}
    P(u, v \in S \wedge \cB_P) = P(u,v \in S \mid \cB_P) P(\cB_P).
  \end{equation}
  Let $X_u$ and $X_v$ be the random variables counting the number of edges in $G_1^*$ that are incident to $u$ and $v$, respectively, disregarding the pairs $\{u,v\}$, $\{u,w_1\}$, and $\{w_{r-1},v\}$. Note that $X_u$ and $X_v$ are independent and have the same distribution as $\Bin(n-3,p_1)$ (or $\Bin(n-2,p_1)$ if $r=1$). Since, conditioned on $\cB_P$, the event $u, v \in S$ implies that $X_u, X_v \leq \dnp + \alpha\sqrt{np\log n} - 1$, it follows that
  \begin{align*}
    P(u,v \in S \mid \cB_P) & \leq P\left( \Bin(n-3, p_1) \leq \dnp + \alpha\sqrt{np \log n} - 1 \right)^2 \\
    & \leq P\left( \Bin(n-1, p_1) \leq \dnp + \alpha\sqrt{np \log n} + 1 \right)^2 \leq 4n^{-1.9},
  \end{align*}
  where the last inequality follows from~\eqref{eq:PvS-two}. Let $\cB$ denote the event that $G_1$ contains two vertices $u, v \in S$ with $\dist_{G_1}(u,v) \leq 4$. For every pair $u,v \in V$ and every $r$, the number of sequences $P$ as above is at most $n^{r-1}$. Applying the union bound over all pairs and all sequences, we get that
  \[
  P(\cB) \leq \binom{n}{2}\sum_{r=1}^4 4n^{r-1}p^rn^{-1.9}  \leq 2\sum_{r=1}^4n^{r-0.9}p^r \leq 8n^{-0.5},
  \]
  where the last inequality follows from the assumption that $p \leq n^{-1+\varepsilon} \leq n^{-0.9}$.
\end{proof}

\subsubsection{Expansion properties of subgraphs of $G_1$}

\label{sec:expander-G1}

In this section, we will show that the graph $G_1$ has good expansion properties even after we delete from it a subgraph $H$ with maximum degree as large as $\delta(G)-2$.

\begin{lemma}
  \label{lemma:G-2-expander}
  There is an absolute positive constant $\eta$ such that a.a.s.~the following statement holds. If $\Gamma$ is a subgraph of $G_1$ with $\Delta(\Gamma) \leq \delta(G) - 2$, then the graph $G_1 \setminus \Gamma$ is an $\left(\eta \sqrt{\frac{n\log n}{p}}, 2\right)$-expander.
\end{lemma}

In the proof of Lemma~\ref{lemma:G-2-expander}, we will use the following proposition, which states that strong expanders retain their expansion properties after we attach to them sets of vertices that have sufficiently large degrees and lie far apart in the enlarged graph.

\begin{prop}[{\cite[Claim~2.8]{BeFeHeKr}}]
  \label{prop:expander-extension}
  Let $G$ be a graph, let $c$ be a positive real, and let $m$ be an integer. Suppose that $S \subseteq V(G)$ satisfies $\deg_G(v) \geq c-1$ for every $v \in S$ and, moreover, there is no path of length at most $4$ in $G$ whose (possibly identical) endpoints lie in $S$. If $G \setminus S$ is an $(m,c)$-expander, then $G$ is an $(m,c-1)$-expander.
\end{prop}
\begin{proof}
  Let $V' = V(G) \setminus S$ and let $G' = G \setminus S = G[V']$. Let $U \subseteq V$ be of cardinality at most $m$, and let $U_1 = U \cap S$ and $U_2 = U \cap V'$. Our assumption on $S$ implies, in particular, that $U_1$ is independent in $G$. It follows that $N_G(U_1) \subseteq V'$. Furthermore, $N_G(U_1)$ can contain at most one vertex from each set $\{v\} \cup N_{G'}(v)$ for every $v \in V'$ and hence $|N_G(U_1) \cap (U_2 \cup N_{G'}(U_2))| \leq |U_2|$. Since
  \[
  N_G(U) = (N_G(U_1) \cup N_{G'}(U_2)) \setminus (N_G(U_1) \cap (U_2 \cup N_{G'}(U_2))),
  \]
  it follows that
  \[
  N_G(U) \geq (c-1)|U_1| + c|U_2| - |U_2| = (c-1)(|U_1| + |U_2|) = (c-1)|U|. \qedhere
  \]
\end{proof}

We will also need the following upper bound on the edge-density of subgraphs of $G(n,p)$ induced by small subsets of vertices.

\begin{lemma}
  \label{lemma:A-upper-density}
  Let $\gamma$ be a positive real. If $p \geq \log n/n$, then a.a.s.~every subset $A$ of vertices of $G(n,p)$ with
  \begin{equation}
    \label{eq:A-size}
    |A| \leq 2 \gamma e^{-2/\gamma-1} \cdot \sqrt{\frac{n\log n}{p}}
  \end{equation}
  satisfies $e(A) \leq \gamma\sqrt{np\log n} \cdot |A|$.
\end{lemma}
\begin{proof}
  Fix a set $A$ of cardinality $a$ satisfying~\eqref{eq:A-size}. Note that $\Ex[e(A)] = \binom{a}{2}p$, let
  \[
  \kappa = \frac{\gamma\sqrt{np \log n} \cdot a}{\binom{a}{2}p} \geq \frac{\gamma \sqrt{np \log n} \cdot a}{a^2p/2} = \frac{2\gamma}{a}\sqrt{\frac{n\log n}{p}},
  \]
  and note that $\kappa \geq e^{2/\gamma+1}$. By Lemma~\ref{lemma:lrg-dev}~(\ref{item:lrg-dev-2}), we have
  \[
    P\left( e(A) > \gamma \sqrt{np \log n} \cdot a \right) \leq \left( \frac{e}{\kappa} \right)^{\gamma \sqrt{np\log n} \cdot a} \leq \exp\left(-2\sqrt{np\log n} \cdot a\right) \leq n^{-2a},
  \]
  where the last inequality follows from our assumption that $p \geq \log n / n$. Applying the union bound over all sets $A$ completes the proof.
\end{proof}

We are finally ready to prove Lemma~\ref{lemma:G-2-expander}.

\begin{proof}[Proof of Lemma~\ref{lemma:G-2-expander}]
  Let $\Gamma$ be an arbitrary subgraph of $G_1$ with $\Delta(\Gamma) \leq \delta(G) - 2$ and let $G' = G_1 \setminus \Gamma$. We first observe that for every $v \in S$, we have
  \[
  \deg_{G'}(v) = \deg_{G_1}(v) - \deg_\Gamma(v) \geq \delta(G_1) - \Delta(\Gamma) = \delta(G) - \Delta(\Gamma) \geq 2.
  \]
  Let $G'' = G'[V \setminus S]$ and let $v \in V \setminus S$. By Lemmas~\ref{lemma:delta} and~\ref{lemma:S}, we may assume that $\delta(G_1) = \delta(G) \leq \dnp$ and that there is no path of length at most $4$ in $G_1$ whose endpoints lie in $S$. In particular, since no two vertices in $S$ have a common neighbor in $G_1$, we have
  \[
  \deg_{G''}(v) \geq \deg_{G'}(v) - 1 \geq \dnp + \alpha\sqrt{np\log n} - \Delta(\Gamma) - 1 \geq \alpha\sqrt{np \log n}.
  \]
  Let $\eta = (\alpha/16)e^{-16/\alpha-1}$ and let $m = \eta\sqrt{n \log n/ p}$. By Lemma~\ref{lemma:A-upper-density} with $\gamma = \alpha/8$, we may assume that every set $A$ of vertices of $G_1$ with $|A| \leq 4m$ satisfies $e_{G_1}(A) \leq (\alpha/8)\sqrt{np\log n} \cdot |A|$. Assume that some $U \subseteq V \setminus S$ satisfies $|N_{G''}(U)| < 3|U|$ and let $A = U \cup N_{G''}(U)$. Since
  \[
  e_{G_1}(A) \geq \frac{1}{2}\sum_{v \in U} \deg_{G_1}(v) \geq \frac{1}{2}\sum_{v \in U} \deg_{G''}(v) \geq \frac{|U|}{2} \cdot \alpha\sqrt{np\log n} > \frac{|A|}{8} \cdot \alpha\sqrt{np\log n},
  \]
  then $|U| > |A|/4 > m$. It follows that every set $U \subseteq V \setminus S$ with $|U| \leq m$ satisfies $|N_{G''}(U)| \geq 3|U|$, i.e., the graph $G''$ is an $(m,3)$-expander. Proposition~\ref{prop:expander-extension} implies that $G'$ is an $(m,2)$-expander.
\end{proof}

\subsubsection{Finding small families of paths covering all vertices of $G$}

\label{sec:paths-in-G1}

In this section, we will show that $G_1^*$ (and hence also $G_1$) a.a.s.~contains a family of $\dh$ pairwise edge-disjoint collections of vertex-disjoint paths covering all vertices of $G$, each of them consisting of at most $n^{1-\lambda}$ paths for some positive constant $\lambda$.

\begin{lemma}
  \label{lemma:paths-in-G1}
  There exists a positive constant $\lambda$ such that a.a.s.~the graph $G_1^*$ contains a collection of pairwise edge-disjoint subgraphs $\cP_1, \ldots, \cP_\dh$, where each $\cP_i$ is a collection of at most $n^{1-\lambda}$ vertex-disjoint paths covering all vertices of $G$.
\end{lemma}

Recall that $G_1^* \sim G(n,p_1)$, $p_1 \geq \log n / (2n)$, and that
\[
np_1 = np - \beta\sqrt{np\log n} \geq \dnp + \sqrt{np\log n}/3 \geq \dnp + \sqrt{np_1\log n}/3,
\]
where the first inequality follows from Lemma~\ref{lemma:delta} and our assumption that $\beta \ll 1$. Since a.a.s. $\delta(G) \leq \dnp$, Lemma~\ref{lemma:paths-in-G1} will be a straightforward corollary of the following statement when we let $c = 1/3$ and $\lambda = 1/72000$.

\begin{lemma}
  \label{lemma:family-of-paths}
  Let $c \in (0, 1/2)$. If $\log n / (2n) \leq p \leq n^{-1+c^2/100}$, then with probability at least $1-o(1)$, the random graph $G(n,p)$ contains a collection of pairwise edge-disjoint subgraphs $\cP_1, \ldots, \cP_k$, where $k = \frac{1}{2}\left(np - c\sqrt{np\log n}\right)$ and each $\cP_i$ is a collection of at most $n^{1-c^2/8000}$ vertex-disjoint paths covering all vertices of $G(n,p)$.
\end{lemma}

Since the proof of Lemma~\ref{lemma:family-of-paths} is somewhat technical, we first give a brief outline of our argument. We will partition the vertex set of $G(n,p)$ into two sets of nearly equal size, denoted $A_1^1$ and $A_2^1$. Each subgraph $\cP_i$ will be the union of two almost perfect matchings: a matching $M_i$ whose every edge is contained in either $A_1^1$ or $A_2^1$ and a matching $N_i$ consisting of edges in the bipartite graph induced by the pair $(A_1^1, A_2^1)$. The existence of a collection of $k$ edge-disjoint matchings $N_1, \ldots, N_k$ in the pair $(A_1^1, A_2^1)$ will follow directly from Lemma~\ref{lemma:k-factor-Gnnp} as every regular bipartite graph can be decomposed into edge-disjoint matchings. In order to construct edge-disjoint matchings $M_1, \ldots, M_k$, we will have to work somewhat harder. We will recursively bisect the sets $A_1^1$ and $A_2^1$ and find our matchings in the obtained nested collection of edge-disjoint bipartite subgraphs (see Claim~\ref{claim:matchings}). The main idea in the proof is that we choose the decomposition $N_1, \ldots, N_k$ randomly in a way that will guarantee that each $N_i$ is contained in a uniform random perfect matching $N'_i$, independently of $M_i$ (see Claim~\ref{claim:Ni}). Such choice will allow us to bound the number of connected components (paths) in the graph $M_i \cup N_i$ (see Claim~\ref{claim:MiNi}).

As we have mentioned above, the proof of Lemma~\ref{lemma:family-of-paths} will rely on the fact (Lemma~\ref{lemma:k-factor-Gnnp}) that with probability $1-n^{-\omega(1)}$, the random balanced bipartite graph $G(n,n;p)$ contains an almost-spanning regular bipartite subgraph of degree almost $np$. Before we prove Lemma~\ref{lemma:k-factor-Gnnp}, we first show that with probability $1-n^{-\omega(1)}$, the random graph $G(n,n;p)$ contains an almost-spanning almost-regular balanced bipartite subgraph of minimum degree almost $np$.

\begin{lemma}
  \label{lemma:balanced-subgraph}
  Let $c \in (0, 1)$. If $4c^2\log n / n \leq p \leq n^{-1+c^2/33}$, then with probability at least $1 - e^{-\sqrt{n}}$, the random bipartite graph $G(n,n;p)$ contains a balanced bipartite subgraph $H$ on $2m$ vertices with $m \geq n - n^{1-c^2/66}$ such that
  \begin{equation}
    \label{eq:degrees-H}
    \delta(H) \geq np - c\sqrt{np\log n} \quad \text{and} \quad \Delta(H) \leq np + c\sqrt{np\log n}.
  \end{equation}
\end{lemma}
\begin{proof}
  Let $A$ and $B$ be the two partite sets of $G(n,n;p)$ and let $z = n^{1-c^2/66}/2$. We will describe an algorithm that constructs a subgraph $H$ with the required properties provided that $G(n,n;p)$ satisfies certain pseudo-random properties that hold with probability at least $1 - e^{-z}$. The algorithm will maintain partitions of $A$ and $B$ into sets $A_+$, $A_0$, $A_-$ and $B_+$, $B_0$, $B_-$, respectively, starting with $A_0 = A$, $B_0 = B$, and the remaining sets empty. Let $H$ be the subgraph of $G(n,n;p)$ induced by the pair $(A_0,B_0)$. We repeat the following two steps until the two conditions in~\eqref{eq:degrees-H} are satisfied or one of the sets $A_+$, $A_-$, $B_+$, $B_-$ contains at least $z$ elements:
  \begin{enumerate}[(1)]
  \item 
    If $\deg_H(v) < np - c\sqrt{np\log n}$ for some $v \in A_0 \cup B_0$, then we move $v$ from $A_0$ ($B_0$) to $A_-$ ($B_-$) and move an arbitrary vertex $v'$ from $B_0$ ($A_0$) to $B_-$ ($A_-$).
  \item 
    If $\deg_H(v) > np + c\sqrt{np\log n}$ for some $v \in A_0 \cup B_0$, then we move $v$ from $A_0$ ($B_0$) to $A_+$ ($B_+$) and move an arbitrary vertex $v'$ from $B_0$ ($A_0$) to $B_+$ ($A_+$).
  \end{enumerate}
  We remark that in both steps the vertex $v'$ is moved in order to guarantee that $|A_-| = |B_-|$ and $|A_+| = |B_+|$. If the procedure stops and $\max\{|A_+|, |A_-|, |B_+|, |B_-|\} < z$, then we are done since the constructed graph satisfies $m = |A_0| = |B_0| \geq n - 2z$,
  \[
  \delta(H) \geq np - c\sqrt{np\log n}, \quad \text{and} \quad \Delta(H) \leq np + c\sqrt{np\log n}.
  \]
  It suffices to bound the probability that $\max\{|A_+|, |A_-|, |B_+|, |B_-|\} \geq z$.

  Consider first the case $|A_-| = |B_-| \geq z$. By construction, this means that the set $A$ contains subsets $A_+$ and $A_-$ of size $z$ each and the set $B$ contains subsets $B_+$ and $B_-$ of size $z$ each such that, letting $A_0 = A \setminus (A_+ \cup A_-)$ and $B_0 = B \setminus (B_+ \cup B_-)$, either $A_-$ contains a set $Z$ of $z/2$ elements such that $\deg(v, B_0) < np - c\sqrt{np\log n}$ for every $v \in Z$ or an analogous statement holds with $A_-$ replaced by $B_-$ and $B_0$ replaced by $A_0$. WLOG we may assume that the former holds, i.e., that $Z \subseteq A_-$. In particular, this set $Z$ satisfies $e(Z,B_0) \leq |Z|(np - c\sqrt{np\log n})$. Fix some $A_-$, $A_+$, $B_-$, $B_+$, and $Z$ as above. Since
  \[
  \Ex[e(Z,B_0)] = p|Z||B_0| \geq p|Z|(n-2z) \geq p|Z|n - (c/2)|Z|\sqrt{np\log n},
  \]
  where the last inequality follows from the fact that $z \leq \frac{c}{4}\sqrt{\frac{n \log n}{p}}$, then by Chernoff's inequality (Lemma~\ref{lemma:lrg-dev}~(\ref{item:lrg-dev-1})),
  \[
  P\left(e(Z,B_0) \leq |Z|\left(np - c\sqrt{np\log n}\right)\right) \leq \exp\left( - \frac{\left(\frac{c}{2}|Z|\sqrt{np\log n}\right)^2}{2p|Z|n} \right) = \exp\left(-\frac{c^2z\log n}{16}\right).
  \]
  Hence, the probability that the procedure stops with $|A_-| = |B_-| \geq z$ can be estimated as follows:
  \begin{equation}
    \label{eq:PAB-}
    P(|A_-| = |B_-| \geq z) \leq 2 \cdot \binom{n}{z}^4 2^z \exp\left(-\frac{c^2z\log n}{16}\right) \leq \left(\frac{2en}{z}\right)^{4z} \exp\left(-\frac{c^2z\log n}{16} \right) \leq e^{-2z},
  \end{equation}
  where the last inequality follows from the fact that $c^2\log n/16 \geq 4\log\left( \frac{2en}{z} \right) + 2$.

  Now, consider the case $|A_+| = |B_+| \geq z$. This means that $A$ contains subsets $A_+$ and $A_-$ of size $z$ each and $B$ contains subsets $B_+$ and $B_-$ of size $z$ each such that, again letting $A_0 = A \setminus (A_+ \cup A_-)$ and $B_0 = B \setminus (B_+ \cup B_-)$, either $A_+$ contains a set $Z$ of $z/2$ elements such that $\deg(v, B_0) > np + c\sqrt{np\log n}$ for every $v \in Z$ or an analogous statement holds with $A_+$ replaced by $B_+$ and $B_0$ replaced by $A_0$. WLOG we may assume that the former holds, i.e., that $Z \subseteq A_+$. In particular, this set $Z$ satisfies $e(Z,B_0) \geq |Z|(np + c\sqrt{np\log n})$. Fix some $A_-$, $A_+$, $B_-$, $B_+$, and $Z$ as above. Since
  \[
  \Ex[e(Z,B_0)] = p|Z||B_0| \leq p|Z|n
  \]
  and $c\sqrt{np\log n} \leq np/2$ by our assumption on $p$, then by Chernoff's inequality (Lemma~\ref{lemma:lrg-dev}~(\ref{item:lrg-dev-1})),
  \[
  P\left(e(Z,B_0) \geq |Z|\left(np + c\sqrt{np\log n}\right)\right) \leq \exp\left( - \frac{\left(c|Z|\sqrt{np\log n}\right)^2}{4p|Z|n} \right) = \exp\left(-\frac{c^2z\log n}{8}\right).
  \]
  It follows from~(\ref{eq:PAB-}) that $P(|A_+| = |B_+| \geq z) \leq e^{-2z}$. This completes the proof.
\end{proof}

\begin{lemma}
  \label{lemma:k-factor-Gnnp}
  Let $c \in (0,1/2)$. If $4c^2\log n / n \leq p \leq n^{-1+c^2/60}$, then with probability at least $1 - e^{-c^4(\log n)^2/36}$ the random bipartite graph $G(n,n;p)$ contains an $(np-c\sqrt{np\log n})$-regular subgraph on $2m$ vertices with $m \geq n - n^{1-c^2/120}$.
\end{lemma}
\begin{proof}
  By Lemma~\ref{lemma:balanced-subgraph} with $c_{\ref{lemma:balanced-subgraph}} = 3c/4$, with probability at least $1 - e^{-\sqrt{n}}$ the graph $G(n,n;p)$ contains a balanced bipartite subgraph $H$ on $2m$ vertices with $\delta(H) \geq np - \frac{3}{4}c\sqrt{np\log n}$ and $\Delta(H) \leq np + \frac{3}{4}c\sqrt{np\log n}$ for some $m$ with $m \geq n - n^{1-c^2/120}$. We show that with probability at least $1 - 3e^{-c^4(\log n)^2/35}$, every such $H$ satisfies the assumptions of Lemma~\ref{lemma:k-factor} with $k = np - c\sqrt{np\log n}$. To see this, denote the color classes of $G(n,n;p)$ by $A$ and $B$ and fix some $X \subseteq A$ and $Y \subseteq B$ with $x = |X| \geq m/140$ and $y = |Y| \geq m/140$. By Chernoff's inequality,
  \begin{align*}
    P\left( e(X,Y) < \frac{kxy}{m} \right) & \leq P\left( e(X,Y) < pxy - \frac{c}{2}\sqrt{\frac{p\log n}{n}}xy \right) \\
    & \leq \exp\left(- \frac{c^2\log n}{8n}xy \right) \leq \exp\left( -\frac{c^2}{10^6}n\log n\right),
  \end{align*}
  where the first inequality follows from the fact that $n-m \leq \sqrt{n/p}$ and the last inequality follows from the fact that $x, y \geq m/140 \geq n/150$. Hence, the probability that $H$ violates assumption~(\ref{item:eXY-lower}) in Lemma~\ref{lemma:k-factor} can be bounded as follows:
  \[
  P(\text{$H$ violates~(\ref{item:eXY-lower})}) \leq 2^{2n}\exp\left(-10^{-6}c^2n\log n\right) \leq e^{-n}.
  \]
  Next, let $D = \frac{\Delta(H)-k}{\delta(H)-k}$ and note that $D \leq 7$. If $H$ violates assumption~(\ref{item:eXY-upper}) in Lemma~\ref{lemma:k-factor}, then $G(n,n;p)$ contains sets $X \subseteq A$ and $Y \subseteq B$ with $x = |X|$ and $y = |Y|$ such that
  \[
  x \leq y = Dx \leq n/20 \text{ and } e(X,Y) > kx \quad \text{ or } \quad  y \leq x = Dy \leq n/20 \text{ and } e(X,Y) > ky.
  \]
  Clearly, those two events have the same probability, so we only need to consider the former. Note that necessarily $y \geq k$ or otherwise $e(X,Y) \leq xy \leq kx$. By Lemma~\ref{lemma:lrg-dev}~(\ref{item:lrg-dev-2}) with $\kappa = ky/p$, for fixed $X$ and $Y$,
  \[
  P(e(X,Y) > kx) \leq \left(\frac{eyp}{k}\right)^{kx} \leq \left( \frac{2Dex}{n} \right)^{\frac{npx}{2}},
  \]
  where the last inequality holds because $k \geq np/2$. Hence, the probability that $H$ violates assumption~(\ref{item:eXY-upper}) in Lemma~\ref{lemma:k-factor} can be bounded as follows:
  \begin{align*}
    P(\text{$H$ violates~(\ref{item:eXY-upper})}) & \leq \sum_{x = k/D}^{n/(20D)} \binom{n}{x}\binom{n}{Dx}\left( \frac{2Dex}{n} \right)^{\frac{npx}{2}} \leq \sum_{x = k/D}^{n/(20D)} \left(\frac{en}{x}\right)^{(D+1)x} \left( \frac{2Dex}{n} \right)^{\frac{npx}{2}} \\
    & \leq \sum_{x = k/D}^{n/(20D)} \left[\left(\frac{2Dex}{n}\right)^{\frac{np}{2}-D-1}\left(2De^2\right)^{D+1}\right]^x \leq \sum_{x = k/D}^{n/(20D)} \left[\left(\frac{2Dex}{n}\right)^{\frac{np}{3}}\left(2De^2\right)^{D+1}\right]^x \\
    & \leq \sum_{x = k/D}^{n/(20D)} \left[\left(\frac{1}{e}\right)^{\frac{np}{3}} \left(2De^2\right)^{D+1} \right]^x \leq \exp\left(-\frac{n^2p^2}{8D}\right) \leq \exp\left(-\frac{c^4(\log n)^2}{35}\right). \qedhere
  \end{align*}
\end{proof}

Finally, we are ready to prove Lemma~\ref{lemma:family-of-paths}.

\begin{proof}[Proof of Lemma~\ref{lemma:family-of-paths}]
  Recall that $k = \frac{1}{2}\left(np - c\sqrt{np\log n}\right)$ and that we are trying to find a collection of $k$ edge-disjoint families of vertex-disjoint paths covering all vertices of $G(n,p)$, such that each collection consists of at most $n^{1-c^2/8000}$ paths. Let us first fix an arbitrary sequence $\cF_1, \ldots, \cF_\ell$ of partitions of the vertex set of $G(n,p)$ with the following properties:
  \begin{enumerate}[(i)]
  \item
    $\cF_i$ contains $2^i + 1$ parts, denoted $A_0^i, A_1^i, \ldots, A_{2^i}^i$ with $|A_1^i| = \ldots = |A_{2^i}^i| = \lfloor 2^{-i}n \rfloor$.
  \item
    For every $i$ and $j$ with $1 \leq i < \ell$ and $1 \leq j \leq 2^i$, we have $A_{2j-1}^{i+1}, A_{2j}^{i+1} \subseteq A_j^i$.
  \item
    \label{item:ell}
    $\ell$ is the smallest integer such that $p2^{-\ell}n < \frac{c}{4}\sqrt{np\log n}$.
  \end{enumerate}
  Let $k_i = p2^{-i}n - \frac{c}{7}\sqrt{p2^{-i}n\log n}$ and let $m_i = 2^{-i}n - (2^{-i}n)^{1-c^2/5880}$.
  \begin{claim}
    \label{claim:Hji}
    With probability $1 - o(1)$, for every $i$ and $j$ with $1 \leq i \leq \ell$ and $1 \leq j \leq 2^{i-1}$, the bipartite subgraph of $G(n,p)$ induced by the pair $(A_{2j-1}^i, A_{2j}^i)$ contains a $k_i$-regular subgraph $H_j^i$ on at least $2m_i$ vertices.
  \end{claim}
  \begin{proof}[Proof of Claim~\ref{claim:Hji}]
    Observe first that by the minimality of $\ell$, we have $p2^{-\ell}n \geq \frac{c}{8}\sqrt{np\log n}$ and hence $2^{-\ell}n \geq \frac{c}{8}\sqrt{\frac{n\log n}{p}} \geq n^{3/4}$. Moreover, for every $i$ with $1 \leq i \leq \ell$,
    \begin{equation}
      \label{eq:p^elln}
      p2^{-i}n \geq p2^{-\ell}n \geq \frac{c}{8}\sqrt{np\log n} \geq \frac{c}{16}\log n \geq 4\left(\frac{c}{7}\right)^2\log\left(2^{-i}n\right)
    \end{equation}
    and $p \leq n^{-1+c^2/2940} \leq \left(2^{-i}n\right)^{-1+c^2/2940}$ and hence by Lemma~\ref{lemma:k-factor-Gnnp}, for every $j$ with $1 \leq j \leq 2^i$, with probability at least $1 - e^{-c^4(\log n)^2/10^5}$, the subgraph of $G(n,p)$ induced by the pair $(A_{2j-1}^i, A_{2j}^i)$ contains a $k_i$-regular subgraph $H_j^i$ on at least $2m_i$ vertices. Since the number of pairs $(i,j)$ as above is at most $\ell2^\ell$, which is no more than $n^{1/4}\log n$, by the union bound, the probability that such $H_j^i$ exist simultaneously for all such $i$ and $j$ is $1-o(1)$.
  \end{proof}
 \begin{claim}
    \label{claim:matchings}
    With probability $1-o(1)$, $G(n,p)$ contains a collection of $k$ edge-disjoint matchings $M_1, \ldots, M_k$ such that for every $s \in \{1, \ldots, k\}$, $M_s$ contains at least $n/2 - n^{1-c^2/5760}$ edges. Furthermore, each of those edges lies inside either $A_1^1$ or $A_2^1$ or, in other words, none of the edges of $M_1 \cup \ldots \cup M_k$ lies in the pair $(A_1^1, A_2^1)$.
  \end{claim}
 \begin{proof}[Proof of Claim~\ref{claim:matchings}]
    Fix an $i$ with $2 \leq i \leq \ell$ and let $H^i = \bigcup_{j = 1}^{2^{i-1}} H_j^i$. Since the graphs $\{H_j^i\}_{j = 1}^{2^{i-1}}$ are pairwise vertex-disjoint, the graph $H^i$ is a $k_i$-regular bipartite subgraph of $G(n,p)$ with at least $2^im_i$ vertices. It follows that $H^i$ decomposes into a collection of $k_i$ edge-disjoint matchings and each of these matchings covers all but at most $n - 2^im_i$ vertices of $G(n,p)$. Moreover, each edge of $H^i$ lies inside one of the sets $A_1^1$ or $A_2^1$. Note that
    \[
    n - 2^im_i \leq 2^i\left(2^{-i}n\right)^{1-c^2/5880} = \left(2^i\right)^{c^2/5880} \cdot n^{1-c^2/5880} \leq n^{1-c^2/7840},
    \]
    where the last inequality follows from the fact that $2^i \leq 2^\ell \leq n^{1/4}$. Since the graphs $H^2, \ldots, H^\ell$ are pairwise edge-disjoint, it suffices to note that
    \[
    \sum_{i = 2}^\ell k_i = np \sum_{i=2}^{\ell}2^{-i} - \frac{c}{7}\sqrt{np\log n} \cdot \sum_{i=2}^\ell2^{-i/2} \geq \frac{np}{2}-np2^{-\ell}-\frac{c}{4}\sqrt{np\log n} \geq \frac{np}{2} - \frac{c}{2}\sqrt{np\log n},
    \]
    where the last inequality follows from the definition of $\ell$, see (\ref{item:ell}).
  \end{proof}
  Let $\cA$ denote the event that the bipartite subgraph of $G(n,p)$ induced by the pair $(A_1^1, A_2^1)$ contains a $k$-regular bipartite subgraph on at least $n - n^{1-c^2/6000}$ vertices and that $G(n,p)$ contains $M_1, \ldots, M_k$ as above. Recall that $P(\cA) = 1 - o(1)$. Conditioning on $\cA$, let $H$ be a uniformly selected random $k$-regular bipartite subgraph of the bipartite subgraph of $G(n,p)$ induced by $(A_1^1, A_2^1)$ with $v(H) \geq n - n^{1-c^2/6000}$. Let $(N_1, \ldots, N_k)$ be a uniformly selected random (ordered) decomposition of $H$ into $k$ matchings, and let $(N_1', \ldots, N_k')$ be a sequence of matchings obtained from $(N_1, \ldots, N_k)$ by randomly extending each $N_i'$ to a perfect matching in $(A_1^1, A_2^1)$, not necessarily using only the edges of $G(n,p)$. The following observation is a crucial step in the proof of Lemma~\ref{lemma:family-of-paths}.
  \begin{claim}
    \label{claim:Ni}
    For each $i \in \{1, \ldots, k\}$, $N'_i$ is the uniform random perfect matching in $(A_1^1, A_2^1)$.
  \end{claim}
  \begin{proof}[Proof of Claim~\ref{claim:Ni}]
    To see this, note that if $N$ and $N'$ are perfect matchings in $(A_1^1, A_2^1)$, then there exists a permutation (relabeling of the vertices) $\phi$ of $A_1^1$ such that $\phi(N) = N'$. Let $H$ be an arbitrary subgraph of $(A_1^1, A_2^1)$ and let $G$ be the (random) subgraph of $G(n,p)$ induced  by the pair $(A_1^1, A_2^1)$. Since $\phi$ can be naturally viewed as a graph isomorphism acting on the set of all bipartite subgraphs of $(A_1^1, A_2^1)$, then $P(G = H) = P(G = \phi(H))$. Moreover, since $\cA$ is a graph property, then $\phi(\cA) = \cA$. Finally, since the definition of $(N_1', \ldots, N_k')$ does not take into the account the labeling of the vertices of $G(n,p)$, it follows that $P(N'_i = N) = P(N'_i = N')$.
  \end{proof}

  Since the definition of $(M_1, \ldots, M_k)$ depends only on the subgraphs of $G(n,p)$ induced by the sets $A_1^1$ and $A_2^1$, the random variables $M_i$ and $N_i'$ (as well as $M_i$ and $N_i$) are independent, i.e., $N_i'$ is a uniform random perfect matching in $(A_1^1, A_2^1)$ even when we fix $M_i$. Clearly, the graph $M_i \cup N_i'$ has maximum degree $2$, i.e., it is a collection of vertex-disjoint paths and cycles. The same is true of $M_i \cup N_i$, which is obtained from $M_i \cup N_i'$ by deleting at most $n^{1-c^2/6000}$ edges. We show that with very high probability, the number of connected components in $M_i \cup N_i$ is at most $n^{1-c^2/8000}$, which implies that $M_i \cup N_i$ contains a collection of at most $n^{1-c^2/8000}$ vertex-disjoint paths covering all vertices of $G(n,p)$.

  \begin{claim}
    \label{claim:MiNi}
    With probability at least $1 - e^{-\sqrt{n}}$, $M_i \cup N_i$ has at most $n^{1-c^2/8000}$ connected components.
  \end{claim}
  \begin{proof}[Proof of Claim~\ref{claim:MiNi}]
    With $M_i$ fixed, consider the following procedure of discovering connected components of $M_i \cup N_i'$ by exposing the edges of $N_i'$ one by one. In the beginning, we mark all vertices of $A_1^1 \cup A_2^1$ as \emph{untouched}. We start exploring a new component by selecting an arbitrary untouched vertex $v$ in $A_1^1 \cup A_2^1$ and marking $v$ as \emph{active}. While there is an active vertex $v$, let $w$ be the neighbor of $v$ in $N_i'$. Observe that since $N_i'$ is a uniform random perfect matching, then $w$ is a uniformly chosen vertex from the untouched vertices in either $A_1^1$ (if $v \in A_2^1$) or $A_2^1$ (if $v \in A_1^1$). Mark both $v$ and $w$ as touched. If $w$ already belongs to the explored connected component (i.e., we close a cycle) or $w$ has no neighbor in $M_i$, then the component is completely discovered. Otherwise, $w$ has an untouched neighbor $w'$ in $M_i$; this $w'$ becomes the new active vertex.
    The key observation is that the number of connected components in $M_i \cup N_i'$ is at most the number of isolated vertices in $M_i$ plus the number of cycles that we close in the above procedure. The number of connected components in $M_i \cup N_i$ is larger by at most $|N_i' \setminus N_i|$. To give a bound on the number of cycles, note that when we expose the neighbor $w$ of the active vertex $v$ and there are still $x$ unexposed edges in $N_i'$ (equivalently, there are $2x$ untouched vertices), then the probability that the edge $vw$ will close a cycle is at most $1 / x$. Assume that at the moment we start exploring a new component, there are still $y$ untouched vertices in each of $A_1^1$ and $A_2^1$. At this stage of our procedure, a cycle will be called \emph{short} if its length is smaller than $y$; otherwise, it is called \emph{long}. Since the probability that we will not close a cycle after exposing the next $z$ edges is at least $\prod_{x = 0}^{z-1} (1-1/(y-x))$, which is at least $(1-1/(y-z))^z$, then the probability that the explored component is a short cycle is at most $1 - (1-2/y)^{y/2}$, which is at most $3/4$, provided that $y \geq 6$. Let $X$ be the random variable denoting the number of short cycles that arise in the above procedure. Since we can close at most $\log_2n$ long cycles before running out of all vertices, then
    \[
    P \left( X > 5\sqrt{n} \right) \leq {5\sqrt{n} \choose \log_2n} \left(\frac{3}{4}\right)^{5\sqrt{n}-\log_2n-6} \leq \left(\frac{3}{4}\right)^{4\sqrt{n}} \leq e^{-\sqrt{n}}.
    \]
    Therefore, with probability at least $1-e^{-\sqrt{n}}$, the number of connected components in $M_i \cup N_i$ can be estimated as follows:
    \[
    \text{\#components in $M_i \cup N_i$} \leq 5\sqrt{n} + n^{1-c^2/6000} + n^{1-c^2/7840} \leq n^{1-c^2/8000}. \qedhere
    \]
  \end{proof}
  Finally, let $\cP_i$ be a subgraph of $M_i \cup N_i$ obtained by removing an arbitrary edge from each cycle in $M_i \cup N_i$. Clearly, with probability at least $1-e^{-\sqrt{n}}$, the number of paths in $\cP_i$ is at most $n - n^{1-c^2/8000}$. Applying the union bound over all $i$ completes the proof.
\end{proof}

\subsection{Turning paths into Hamilton cycles}

\label{sec:merging-paths}

In this section, we will show how, using the few random edges that we have put aside in $G_2$, we can convert the collections of paths $\cP_1, \ldots, \cP_\dh$ into edge-disjoint Hamilton cycles. To this end, we will further split $G_2$ into $\dh$ random graphs $G_{2,1}, \ldots, G_{2,\dh}$ and for each $i$, we will alter $\cP_i$ using only the edges of $G$ and $G_{2,i}$. Define $p_3$ by $(1-p_3)^\dh = 1-p_2$ and note that
\[
p_3 \geq \frac{p_2}{\dh} \geq \frac{p_2}{np} \geq \beta \sqrt{\frac{\log n}{n^3p}}.
\]
For each $i$, let $G_{2,i}$ be the binomial random graph on the vertex set $V \setminus S$ with edge probability $p_3$. Note that $G_2$ has the same distribution as $\bigcup_{i=1}^\dh G_{2,i}$ and hence $G$ has the same distribution as $G_1 \cup \bigcup_{i=1}^\dh G_{2,i}$.

Fix an $i$ with $1 \leq i \leq \dh$ and suppose that we have already converted $\cP_1, \ldots, \cP_{i-1}$ into edge-disjoint Hamilton cycles $C_1, \ldots, C_{i-1}$ using only the edges of $G_1$ and $G_{2,1}, \ldots, G_{2,i-1}$ and no edges of $\bigcup_{j \geq i} \cP_j$. Let $\Gamma_i = \bigcup_{j < i}C_j \cup \bigcup_{j > i} \cP_j$. To complete the proof, it suffices to show that we can convert $\cP_i$ into a Hamilton cycle using only the edges of $(G_1 \cup G_{2,i}) \setminus \Gamma_i$. In order to do that, we further split the graph $G_{2,i}$ into $n^{1-\lambda}$ binomial random graphs $G_{2,i,1}, \ldots, G_{2,i,n^{1-\lambda}}$ on the vertex set $V \setminus S$ with edge probability $p_4$ defined by $(1-p_4)^{n^{1-\lambda}} = 1-p_3$. Clearly, $G_{2,i}$ has the same distribution as $\bigcup_{s=1}^{n^{1-\lambda}}G_{2,i,s}$ and
\[
p_4 \geq \frac{p_3}{n^{1-\lambda}} \geq \beta \sqrt{\frac{\log n}{n^{5-2\lambda}p}}.
\]
For every integer $s$ with $0 \leq s \leq n^{1-\lambda}$, let $G_s' = (G_1 \cup \bigcup_{j > s}G_{2,i,j}) \setminus \Gamma_i$. Fix an $s$ with $1 \leq s \leq n^{1-\lambda}$ and assume that $G_s'$ contains a collection $\cP$ of $s$ vertex-disjoint paths that cover all vertices in $V$. In the key step of the proof, we show that with probability at least $1 - n^{-4}$, the graph $G_{s-1}'$ contains either a collection of $s-1$ vertex-disjoint paths that cover all vertices in $V$ or a Hamilton cycle (if $s = 1$).

We give a brief outline of our argument. First, we will make an extremal choice of $\cP$ that will guarantee that some longest path in $\cP$, denoted $P_1$, cannot be further extended using the edges of $G_s'$. Since $G_s'$ is a good expander, Lemma~\ref{lemma:Posa} will imply that either the graph $G_s'[V(P_1)]$ is Hamiltonian or it contains many boosters (see Definition~\ref{dfn:booster}). In the latter case, with high probability one of those boosters will be an edge of the random graph $G_{2,i,s}$, which is independent of $G_s'$. It will follow that with high probability the graph $G_{s-1}'[V(P_1)]$ is Hamiltonian. If $P_1$ is the only path in $\cP$, then we will be done. Otherwise, we can either merge the second longest path, denoted $P_2$, with the Hamilton cycle in $G_{s-1}'[V(P_1)]$ using some edge of $G_s'$, or an argument analogous to the one given above for $P_1$ will show that with high probability $G'_{s-1}[V(P_2)]$ is Hamiltonian. In the latter case, with high probability $G_{s-1}'$ will contain an edge joining the two cycles spanning $V(P_1)$ and $V(P_2)$; this edge can be used to merge those two cycles into a path spanning $V(P_1) \cup V(P_2)$.

To formalize the above discussion, we introduce a partial order $\lep$ on the set of all families of paths covering all vertices in $V$.

\begin{dfn}
  Let $\cP$ and $\cP'$ be two families of paths covering all vertices in $V$. Assume that $\cP$ and $\cP'$ consist of paths $P_1, \ldots, P_s$ and $P_1', \ldots, P_{s'}'$, respectively, where $|P_1| \geq \ldots \geq |P_s|$ and $|P_1'| \geq \ldots \geq |P_{s'}'|$. We say that $\cP \lep \cP'$ if
  \begin{enumerate}[(i)]
  \item 
    $s < s'$, i.e., $\cP$ consists of fewer paths than $\cP'$ or 
  \item
    $s = s'$ and $(|P_1|, \ldots, |P_s|)$ is lexicographically larger than $(|P_1'|, \ldots, |P_s'|)$, i.e., there is some $r$ with $1 \leq r \leq s$ such that $|P_r| > |P_r'|$ and $|P_{r'}| = |P_{r'}'|$ for all $r'$ with $r' < r$.
  \end{enumerate}
\end{dfn}

\noindent
WLOG we may assume that $\cP$ is a $\lep$-minimal spanning collection of paths in $G_s'$ and consists of paths $P_1, \ldots, P_s$. This extremal choice of $\cP$ has the following implication that will make our later analysis much clearer and easier:

\begin{claim}
  \label{claim:extremal-choice}
  For every $r$ with $1 \leq r < s$, each endpoint of every Hamilton path in $G_s'[V(P_r)]$ has no neighbors in $\bigcup_{r' > r} V(P_{r'})$.
\end{claim}
\begin{proof}
  Suppose that $G_s'[V(P_r)]$ contains a Hamilton path $v_0 \ldots v_{\ell_r}$ such that $v_{\ell_r}$ has a neighbor in $V(P_{r'})$ for some $r' > r$. Assume that $P_{r'} = w_0 \ldots w_{\ell_{r'}}$ and let $j$ be such that $v_{\ell_r} w_j$ is an edge of $G_s'$. Replacing the paths $v_0 \ldots v_{\ell_r}$ and $w_0 \ldots w_{\ell_{r'}}$ in $\cP$ with paths $v_0 \ldots v_{\ell_r} w_j \ldots w_0$ and $w_{j+1} \ldots w_{\ell_{r'}}$ yields a family of paths $\cP' \subseteq G_s'$ satisfying $\cP' \lep \cP$, which contradicts the choice of $\cP$.
\end{proof}

Recall the definition of $\Gamma_i$ and note that $\Delta(\Gamma_i) \leq 2(\dh-1) \leq \delta(G)-2$. Let $\eta$ be the constant from the statement of Lemma~\ref{lemma:G-2-expander} and let $m = \eta\sqrt{n \log n/p}$. It follows from Lemma~\ref{lemma:G-2-expander} that the graph $G_s'$, which contains $G_1 \setminus \Gamma_i$ as a subgraph, is an $(m,2)$-expander. Let $V_1 = V(P_1)$. The following statement is the core of our argument.
\begin{claim}
  \label{claim:V_1}
  With probability at least $1 - n^{-5}$, the graph $G_{s-1}'[V_1]$ is Hamiltonian and $|V_1| \geq m$.
\end{claim}
\begin{proof}
  By Claim~\ref{claim:extremal-choice} and Lemma~\ref{lemma:Posa}, either $|V_1| \geq m$ and $G_s'[V_1]$ is Hamiltonian or $G_s'[V_1]$ contains at least $m^2/2$ boosters. Denote the set of those boosters by $E_1$ and let $E_1'$ be all the pairs in $E_1$ that are not edges of $\Gamma_i$ and are fully contained in the set $V \setminus S$. Observe that
  \[
  |E_1'| \geq |E_1| - |S|n - \Delta(\Gamma_i)n \geq \frac{m^2}{2} - n^{1.1} - n^2p \geq \frac{m^2}{3},
  \]
  where the second inequality follows form Lemma~\ref{lemma:S} and the last inequality follows from our assumption that $p \leq n^{-1+\varepsilon} \leq n^{-1/2}$. By the definition of a booster, if any pair in $E_1'$ is an edge of $G_{2,i,s}$, then $G_{s-1}'[V_1]$ is Hamiltonian. Therefore
  \begin{align*}
    P\left( \text{$G_{s-1}'[V_1]$ is not Hamiltonian} \right) & \leq (1-p_4)^{m^2/3} \leq e^{-p_4m^2/3} \leq \exp\left( - \frac{\beta}{3}\sqrt{\frac{\log n}{n^{5-2\lambda}p}} \cdot \frac{\eta^2n\log n}{p}\right) \\
    & \leq \exp\left( - \frac{\beta\eta^2}{3} p^{-3/2}n^{-3/2+\lambda}\right) \leq \exp\left( -\frac{\beta\eta^2}{3}n^{\lambda/2} \right) \leq n^{-5},
  \end{align*}
  where the second to last inequality follows from the assumption that $p \leq n^{-1+\varepsilon} \leq n^{-1+\lambda/3}$.
\end{proof}

If $s = 1$, then there is nothing left to prove, so we may assume that $s \geq 2$. Let $V_2 = V(P_2)$. If $G_s'[V_2]$ contains a Hamilton path $P'$, one of whose endpoint has a neighbor in $V_1$, then we are done, since we can replace $P_1$ and $P_2$ in $\cP$ with the path spanning all vertices in $V_1 \cup V_2$ that we obtain from merging the Hamilton cycle in $G_{s-1}'[V_1]$ and the path $P'$. Otherwise, by Claim~\ref{claim:extremal-choice} and Lemma~\ref{lemma:Posa}, $|V_2| \geq m$ and $G_s'[V_2]$ is either Hamiltonian or it contains at least $m^2/2$ boosters. Denote this set of boosters by $E_2$. Similarly as in the proof of Claim~\ref{claim:V_1}, let $E_2'$ be the set of all pairs in $E_2$ that are not edges of $\Gamma_i$ and are contained in $V \setminus S$ and observe that $|E_2'| \geq m^2/3$. Since the sets $E_1'$ and $E_2'$ are disjoint, then independently of $G_{s-1}'[V_1]$, we have
\[
P\left( \text{$G_{s-1}'[V_2]$ is not Hamiltonian} \right) \leq n^{-5}.
\]
Finally, we have seen that with probability at least $1 - 2n^{-5}$, either the graph $G_{s-1}'$ contains a collection of $s-1$ paths covering all vertices in $V$ or the graphs $G_{s-1}'[V_1]$ and $G_{s-1}'[V_2]$ are Hamiltonian and $|V_1|, |V_2| \geq m$. In the latter case, let $E_3$ be the set of all pairs $uv$ with $u \in V_1$ and $v \in V_2$ and let $E_3'$ be all the pairs of $E_3$ that are not edges of $\Gamma_i$ and are contained in $V \setminus S$. Observe that $|E_3'| \geq m^2/3$ and that if any pair in $E_3'$ is an edge of $G_{2,i,s}$, then the graph $G_{s-1}'$ contains a collection $\cP'$ of $s-1$ paths covering all the vertices in $V$. We obtain such a collection by merging the two cycles spanning $V_1$ and $V_2$ with an arbitrary edge of $E_3'$. Since the set $E_3'$ is disjoint from $E_1'$ and $E_2'$, the probability that no pair in $E_3'$ is an edge of $G_{s-1}'$ is at most $n^{-5}$, independently of $G_{s-1}'[V_1]$ and $G_{s-1}'[V_2]$.

To summarize, we have shown that, conditioned on the existence of a family of $s$ vertex-disjoint paths in $G_s'$ covering all vertices of $G$, with probability at least $1 - 3n^{-5}$, the graph $G_{s-1}'$ contains a family of $s-1$ such paths or a Hamilton cycle (if $s = 1$). Since the graph $G_{n^{1-\lambda}}'$ contains a collection of at most $n^{1-\lambda}$ such paths, it follows that with probability at least $1 - n^{-4}$, the graph $G_0'$ is Hamiltonian. Finally, by the union bound, we conclude that, conditioned on the existence of $\cP_1, \ldots, \cP_\dh$ in $G_1$, with probability at least $1 - n^{-3}$, $G$ contains a collection of $\dh$ edge-disjoint Hamilton cycles. Since we have already shown that a.a.s.~we can find such $\cP_1, \ldots, \cP_\dh$, the proof is complete.

\bigskip
\noindent
{\bf Acknowledgment.}
The authors would like to thank Sonny Ben-Shimon for helpful discussions.

\bibliographystyle{myamsplain}
\nocite{*}
\bibliography{HC-packing}

\end{document}